\documentclass[12pt]{amsart}
\usepackage{dsfont}
\usepackage{graphicx}
\usepackage{amsmath}
\usepackage{color}
\usepackage{comment}

\newtheorem{Thm}[equation]{Theorem}
\newtheorem{Cor}[equation]{Corollary}
\newtheorem{Lem}[equation]{Lemma}
\newtheorem{Pro}[equation]{Proposition}

\theoremstyle{definition}

\newtheorem{Def}[equation]{Definition}

\theoremstyle{remark}

\newtheorem{Rem}[equation]{Remark}

\numberwithin{equation}{section}
\makeatletter
\renewcommand{\c@figure}{\c@equation}
\makeatother

\newcommand{\M}{{\mathcal M}}

\newcommand{\vol}{d{\mbox{\rm vol}}}

\newcommand{\eop}[1]{{\flushright\hfill\fbox{\bf #1}}}

\newcommand{\ack}{\noindent{\bf Acknowledgement.}}

\title[Poincar\'e inequality]{Poincar\'e inequality on complete Riemannian manifolds with Ricci curvature bounded below}

\author{G\'erard Besson}
\address{Institut Fourier\\ 
Universit\'e Grenoble Alpes\\ 
Institut Fourier\\
100 rue des maths, 38610 Gi\`eres}
\email{g.besson@univ-grenoble-alpes.fr}

\author{Gilles Courtois}
\address{Department of Mathematics\\Paris VI\\4 place Jussieu, 75232 Paris C\'edex 09}
\email{gilles.courtois@imj-prg.fr}

\author{Sa'ar Hersonsky}
\address{Department of Mathematics\\ 
University of Georgia\\ 
Athens, GA 30602}
\email{saarh@uga.edu}

\thanks{}
\date{\today}

\begin{document}

\begin{abstract}
We prove that complete Riemannian manifolds with
polynomial growth and Ricci curvature bounded from below, admit uniform Poincar\'e inequalities. A global, uniform Poincar\'e inequality for 
horospheres in the universal cover of a closed, $n$-dimensional Riemannian manifold with pinched negative sectional curvature follows as a corollary.
 \end{abstract}

\maketitle

\section{Introduction}
\label{se:Intro}
\paragraph{\it Statements of the main results.}  In this paper, we will establish that complete Riemannian manifolds with Ricci curvature bounded below and having polynomial growth, admit a family of uniform Poincar\'e inequalities. To begin, 
let $(M^n,g)$ be a complete $n$-dimensional Riemannian manifold. Henceforth, we will assume  that $(M^n,g)$ satisfies the Ricci curvature lower bound 
\begin{equation}
\label{eq:boundonRicci}
{\rm Ricci}_{(M^{n},g)} \geq -(n-1)\kappa  , \ {\mbox {\rm for some}}\  \kappa\geq 0.
\end{equation}

We will also assume that  $(M^n,g)$ has $\alpha$-polynomial growth; this means that there exist constants $v>0$,
$\alpha >0$ and $R_0\geq 0$ such that for any $m\in M^n$ and $R>R_0$, the ball of radius $R$ centered at $m$ 
satisfies 
\begin{equation}
\label{eq:polygrowth}
{\mbox{\rm vol}} \, B(m,R) \leq v R^{\alpha},
\end{equation} 
where vol denotes the canonical measure on $(M^n,g)$. Recall that  
Bishop's Comparison Theorem (cf. \cite[Section IV]{Sa}) implies  that when $\kappa = 0$, 
$(M^n,g)$ satisfies polynomial growth with $\alpha=n$. 

\smallskip

The triple 
$(M^n,{\rm dist},\mbox{{\rm vol}})$ with ${\rm dist}$ being the standard metric induced by the Riemannian metric  is an example of a {\it metric measure space}.  Throughout this paper $(X,\rho,\mu)$ will denote a metric space endowed with a Borel measure $\mu$. 
We will use the notation 
\begin{equation}
\label{eq:average}
u_A= \frac{1}{\mu(A)}\int_A u d\mu,
\end{equation}
for every $A\subset X$ and measurable function $u:X\rightarrow [-\infty,\infty]$,
and when $A$ is a ball $B(m,R)$, we will abuse the notation and write 
\begin{equation}
\label{eq:ballaverage}
u_R= \frac{1}{\mu(B(m,R))}\int_{B(m,R)} u d\mu.
\end{equation}
\medskip

\medskip

We will say that a Riemannian manifold $(M^n,g)$ satisfies 
a $(\sigma, \beta, \sigma)$-{\it uniform} (with respect to balls) Poincar\'e inequality, for  $\sigma \geq 1$, if there exists a constant
$C$ such that for any $r_0>0$, there exists
a constant $K$ such that for any $u\in C^{1}(M^n)$, any $R \geq r_0$
and any ball $B(m,R)\subset M^n$, we have  
\begin{equation}
\label{sigma-beta-sigma-P}
\int _{B(m,R)} |u(x)- u_R |^{\sigma}d{\mbox{\rm vol}}(x)  \leq K R^{\beta} \int _{B(m,CR)} |\nabla u (x)|^{\sigma} d{\mbox{\rm vol}}(x).
\end{equation}
Our main result is the following: 

\begin{Thm}[Main Theorem]
\label{th:main1}
Let $(M^n,g)$ be a complete Riemannian manifold satisfying the Ricci curvature bound (\ref{eq:boundonRicci})
and the $\alpha$-polynomial growth assumption (\ref{eq:polygrowth}).
Then, there exists a constant $C=C(n,\kappa)$ such that for any $\sigma \geq 1$ and $r_0>0$, there exists
a constant $K=K(n,\sigma, r_0, R_0, \kappa, v)$ such that for any $u\in C^{1}(M^n)$, any $R \geq r_0$
and any ball $B(m,R)\subset M^n$, we have
\begin{equation}
\label{P1}
\int _{B(m,R)} |u(x)- u_R |^{\sigma}d{\mbox{\rm vol}}(x)  \leq K R^{\alpha + \sigma -1} \int _{B(m,CR)} |\nabla u (x)|^{\sigma} d{\mbox{\rm vol}}(x),
\end{equation}  
where $u_R= u_{B(m,R)}$.
\end{Thm}

This theorem is meaningful for large balls. Indeed, since balls of small radii in $(M^n,g)$ are asymptotically Euclidean, they carry
$(\sigma, \sigma, \sigma)$-uniform Poincar\'e inequalities for $\sigma \geq 1$. This is the reason for the restriction to balls of radius $R\geq r_0$, and it also explains why the constant
$K$ depends, among other geometric quantities, on $r_0$.

\smallskip

The exponent of $R$ in the above Theorem is optimal. For every $\alpha \in \mathbb N$, $\alpha \geq 1$, we construct examples of complete Riemannian manifolds $(M^n, g)$ with 
$\alpha$-polynomial growth 
${\mbox{\rm vol}} \, B(m,R) \leq v R^{\alpha}$ and  Ricci curvature bounded below, $\rm{Ricci} \geq -(n-1)\kappa $ such that there exist a function $u: M^n \rightarrow \mathbb R$ such that  
for any constants $C>0$, $\sigma \geq 1$,  and any $\beta < \alpha + \sigma -1$, 
\begin{equation}
\label{eq:optimal1}
\lim _{R\to \infty} \left( \int _{B(R)} |u- u_{R} |^{\sigma}\right)\left(R^{\beta}\int _{B(CR)} |\nabla(u)|^{\sigma} \right)^{-1} = \infty,
\end{equation}
see section \ref{se:example}.

\begin{Rem}
Note that the constants of the Poincar\'e inequality are uniform among the set of
all Riemannian manifolds with the same Ricci curvature lower bound and polynomial growth of order $\alpha$.
\end{Rem}

Next we apply our main theorem to special types of hypersurfaces. 
Let $\tilde{M}$ be the universal cover of  a closed, complete, $n$-dimensional Riemannian manifold, $(M^n,g)$,  whose sectional curvature satisfies
\begin{equation}
\label{seccurv}
-b^2\leq K_g\leq -a^2,
\end{equation}
where $a,b$ are fixed positive constants. We will show that \emph{horospheres} in $\tilde M$ satisfy a uniform and {\it global}  
Poincar\'e inequality (\ref{sigma-beta-sigma-P}), where global means that the inequalities hold {\it independently} of the horosphere. 
\begin{Cor}
\label{horospheres}
There exist positive constants $C$ and $K$, as in Theorem~\ref{th:main1}, depending only on 
$n$, $\sigma$, $a$ and $b$, such that every horosphere ${\mathcal H}$ in
$\tilde M$, endowed with the induced Riemannian metric, satisfies inequality $(\ref{P1})$  with 
\begin{equation}
\label{alpha}
\alpha= \frac{(n-1)b}{a}.
\end{equation}
\end{Cor}

\medskip

\paragraph{\it Perspective.}  Poincar\'e inequalities are central in the study of the geometrical analysis of manifolds. 
It is well  known that carrying  a  Poincar\'e inequality has strong geometric consequences. For instance, 
a complete, doubling, non-compact, Riemannian manifold admitting a $(1,1,1)$-uniform Poincar\'e inequality 
satisfies an isoperimetric inequality. Moreover, a $(2,2,2)$-uniform Poincar\'e inequality is equivalent to Gaussian estimates
for the heat kernel, \cite{Co-Sa}, \cite{Gri-Sa}. For comprehensive 
and detailed accounts of the subject, the reader is advised to consult for example \cite{HaKo} and  \cite{Sal}.

\smallskip

To put our main theorem in perspective, let us recall  a few classical results. The following theorem follows, for instance, from Buser's inequality, \cite{Bu} (and an application of Minkowski inequality).  

\begin{Thm}
\label{cor-sigma-P_loc}
Assume that $M^n$ has Ricci curvature bounded below by $-(n-1)\kappa$. Then, 
for all $R>0$ and for all $\sigma\geq1$, there exists a constant $C(n, R, \kappa )$ such that
for all $u\in C^1(M)$ and for all  $m\in M$,  we have 
\begin{equation}
\label{eq:localP}
\int _{B(m,R)} |u(x)-u_R|^{\sigma} d {\rm vol}(x)
\leq C(n, R,\kappa )\,  \int _{B(m,3R)} |\nabla u (x)|^{\sigma} d{\rm vol}(x), 
\end{equation}
where 
$u_R=u_{B(m,R)}$.
\end{Thm}
A manifold for which (\ref{eq:localP}) holds is said to carry a \emph{local} Poincare inequality. For the proof see \cite[Chapter VI.5]{Ch}, or \cite[Lemma 2.9]{hein} for a different proof based on the Cheeger-Colding segment inequality, or  \cite[Theorem 5.6.6]{Sal}.
In fact, 
by \cite[Section 10.1]{HaKo}, and under the same assumptions, the following holds:
\begin{equation}
\label{eq:bu}
\int_{B(m,R)}|u-u_R| d{\mbox{\rm vol}} \leq C(n) \exp{((n-1)\kappa R)}\, R \int_{B(m,R)}|\nabla u|d {\mbox{\rm vol}}.
\end{equation}
Hence, when  $(M^n,g)$ has non-negative Ricci curvature, (i.e., a bound (\ref{eq:boundonRicci}) with $\kappa =0$), then for every $\sigma \geq 1$, 
$(M^n,g)$ satisfies a  $(\sigma, \sigma, \sigma)$-uniform Poincar\'e inequality: For every $R>0$, 
\begin{equation}
\label{positivekappa-sigma-sigma-P}
\int _{B(m,R)} |u- u_R |^{\sigma}d{\mbox{\rm vol}}  \leq K(n,\sigma) R^{\sigma} \int _{B(m,2R)} |\nabla u |^{\sigma} d{\mbox{\rm vol}}
\end{equation}
(cf. \cite[Theorem 5.6.6]{Sal}).

\begin{Rem}
The constants which appear above do not depend on the point $m$.
\end{Rem}


Another important class of examples occur when 
 $(M^n,g)$ is a unimodular connected Lie group equipped with a left invariant metric. Then, for every $\sigma \geq 1$,  the following Poincar\'e inequality is known to hold (cf. \cite[page 173]{Sal}):
\begin{equation}
\label{eq:bu1}
\int_{B(m,R)}|u-u_R|^\sigma  d{\mbox{\rm vol}} \leq (2R)^\sigma   \frac{\mathrm{vol} \, B(2R)}{\rm{vol}\, B(R))}  \int_{B(m,R)}|\nabla u| ^\sigma d {\mbox{\rm vol}}.
\end{equation}
Moreover, if  the left invariant metric is doubling, (i.e., if  there exists a constant $C$ 
such that for any $R>0$, $\rm{vol} \,B(2R) \leq C\, \rm{vol} \, B(R)$, 
then for every $\sigma \geq 1$, such a Lie group satisfies a $(\sigma, \sigma, \sigma)$-uniform Poincar\'e inequality \cite[Theorem 5.6.1]{Sal}:
 \begin{equation}
\label{eq:bu2}
\int_{B(m,R)}|u-u_R|^\sigma  d{\mbox{\rm vol}} \leq C (2R)^\sigma  \int_{B(m,R)}|\nabla u| ^\sigma d {\mbox{\rm vol}}.
\end{equation}

Lie groups equipped with doubling left invariant metric have polynomial growth and examples of such 
groups are nilpotent ones. In \cite{Kl}, Kleiner proved analogous Poincar\'e inequalities to (\ref{eq:bu1}) and (\ref{eq:bu2}) for discrete 
finitely generated groups. Besides manifolds of non-negative Ricci curvature, unimodular Lie groups and discrete finitely generated groups with doubling property or
manifolds which are roughly isometric to these (cf. definition (\ref{de:roughisometry})); no other class of manifolds 
are known to satisfy $(\sigma, \sigma, \sigma)$-uniform Poincar\'e inequality.

\medskip

\paragraph{\it The scheme of the proof of our main theorem and the structure of this paper.}
\  In \cite{Co-Sa}, foundational work of  Coulhon and Saloff-Coste shows that under two conditions on $(M^n,g)$,  it admits  a uniform Poincar\'e inequality (\ref{sigma-beta-sigma-P}) \emph{if and only if} a graph approximation of $(M^n,g)$ admits a {\it discrete} version of this inequality (see Definition~\ref{de:discreteP}).  The first of these conditions is a  local Poincar\'e inequality  (\ref{eq:localP}) which, as we mentioned before,  holds under the lower bound assumption of the  Ricci curvature.    The second condition in \cite{Co-Sa} is a local doubling property.  It follows from the  lower bound assumption on the Ricci curvature, that this property holds on $(M^n,g)$ as well as on any of its graph approximations.
Thus,  it is sufficient to  prove a Poincar\'e inequality for any graph approximation of $(M^n,g)$. 



\smallskip
Section~\ref{se:se2} is devoted to a detailed exposition of the part of the work in \cite{Co-Sa} that we need in this paper.
In Section~\ref{se:discretization},  we describe  a discretization scheme which is inspired by the seminal works of Kanai (cf.  \cite{Kan1,Kan2}). Kanai's scheme provides a bounded valence graph approximation of $(M^n,g)$. His scheme relies on the Ricci curvature lower bound  assumption and requires in addition positivity of the injectivity radius of $(M^n,g)$. However,  Coulhon and Saloff-Coste   \cite{Co-Sa} later made the important observation that a local doubling volume assumption is a sufficient one. In this section, we will also recall  a theorem of Kanai 
and its improvement by Coulhon-Saloff-Coste relating 
the growth rate of the manifold and the growth of any of its graph approximation: 
the graph and the manifold are roughly quasi-isometric as metric measure spaces (see Definition~\ref{de:roughisometry}), and it follows that the $\alpha$-polynomial growth property transfers from the manifold to any of its graph approximation. 
\smallskip

In Section~\ref{se:graphs}, we prove a discrete version of Poincar\'e inequality (\ref{sigma-beta-sigma-P}), which is a slight generalization of the one given in  \cite[308--311]{Co-Sa2}, which holds for graph having polynomial growth. In Section~\ref{se:proofs}, 
we assemble the pieces together and provide the proofs of our main theorem  and of Corollary \ref{horospheres}. It is in this part that the  assumption of the polynomial growth of the manifold is first used. Finally, In Section~\ref{se:example}, 
we present examples elucidating the sharpness of our inequalities. 
 
\smallskip




\noindent{\bf Notation.}
Henceforth, we will let $M$ denote $(M^n,g)$ and we let $d$ denote the {\rm dist} function on $M$.

\smallskip

\ack\  The research in this paper greatly benefited from visits of the authors at Institut Fourier, IHP, Paris VI, Princeton University and the University of Georgia. The authors express their gratitude for their hospitality.
We would also like to deeply thank Tobias Colding and Laurent Saloff-Coste for their advice and insight regarding the 
subject of this paper. G.~Besson is supported by ERC Advanced Grant 320939, GETOM. S.~Hersonsky is supported by
grant 319163 from the Simons Foundation.

\medskip

\section{Discretization of riemannian manifolds}
\label{se:discretization}
In this section, we recall some basic definitions and lemmas that are needed before we state the main application that is needed in this paper, Corrollary~\ref{poly-discretization}. This corollary relates the property of polynomial growth of a manifold to any of a tight discrete approximation of it. Throughout this section, we will closely follow the notation and logic as in  \cite{Co-Sa} and in 
\cite{Kan1,Kan2}.
\begin{Def}[\cite{Kan1,Kan2}]
\label{de:roughisometry}
Let  $(X,\rho, \mu)$ and $(Y,d,\nu)$
be two metric measure spaces. 
A map 
$\phi : X \to Y$ is called a {\it rough isometry} if there exist constants $c_1>0$ and $c_2,c_3>1$ such that
 \begin{equation}\label{rough1}Y=\cup_{x\in X}B(\phi(x),c_1),\end{equation}
\begin{equation}\label{rough2}c_2^{-1}( \rho(x,y)-c_1)\leq d(\phi(x),\phi(y))\leq c_2(\rho(x,y)+c_1),\ \mbox{\rm and}\end{equation}
\begin{equation}\label{rough3}c_3^{-1}\mu(B(x,c_1))\leq \nu(B(\phi(x),c_1))\leq c_3\mu(B(x,c_1)).\end{equation}
If there exists a rough isometry between two metric measure spaces, they are said to be roughly isometric.
\end{Def}

\medskip

 Let $M$ be a complete Riemannian manifold. A subset ${\mathcal G}$ of $M$ is said to be $\epsilon$-{\it separated}, for $\epsilon>0$, if the riemannian distance between any two distinct points of ${\mathcal G}$ is greater than or equal to $\epsilon$. An $\epsilon$-{\it discretization} of $M$ is an $\epsilon$-separated subset  ${\mathcal G}$ of $M$ which is {\it maximal} with respect to inclusion of sets.  The maximality  implies that 
\begin{equation}
\label{discretizationcover}
M=\bigcup_{\xi\in {\mathcal G}}B(\xi,\epsilon),  
\end{equation}
and $\epsilon$ is then called the {\it covering radius} of the discretization.
 The graph structure on ${\mathcal G}$ is determined by defining the neighbors of $\xi\in{\mathcal G}$ to be the set 
\begin{equation}
\label{eq:graphstructure}
{\bold N}(\xi)=\{ {\mathcal G}\cap B(\xi,2\epsilon\}\setminus\{\xi\}.
\end{equation}

The multiplicity $\mathcal{M}(\mathcal{G}, \epsilon )$ of the covering $M=\bigcup_{\xi\in {\mathcal G}}B(\xi,\epsilon )$ is defined by
\begin{equation}\label{multi}
\mathcal{M}(\mathcal{G}, R)= \rm{sup}_{\xi \in \mathcal G}| {\bold N}(\xi)|.
\end{equation}

\medskip

In fact, this graph which we denote by
$X$,  carries a structure of a metric-measured space  $(X,\rho,\mu)$: The distance $\rho$ on $X$ is the canonical combinatorial distance, and the measure $\mu$ is defined  by 
\begin{equation}
\label{eq:measureonpoints}
\mu(x)= {\rm vol}(B(x,\epsilon)),\ \text{ for each } x\in X. 
\end{equation}
The following definition allows one to distinguish metric measure spaces with a special property of the measure.
\begin{Def}
\label{loc-doub}
A metric measured space $(X, \rho , \mu)$ satisfies  the {\it local doubling condition}, $(DV)_{loc}$, 
if for all $r>0$ there exists 
$C_r$ such that for all $x\in X$  
\begin{equation}\label{D_loc}
 \mu (B(x,2r)) \leq C_r \mu (B(x,r)),
\end{equation}
where the constant $C_r$ depends on $r$ but is independent of the point $x$.
\end{Def}

\bigskip

The following lemma, which will be used in the proof of Theorem~\ref{discretization}, asserts that the assumption of  local doubling  implies uniform control on the multiplicity of the covering $\{ B(x,3\epsilon)\} _{x\in X}$. It was first proved by Kanai  \cite[Lemma 2.3]{Kan1} under a Ricci lower bound curvature assumption. However, it turns out that the main ingredient in Kanai's proof is the $(DV)_{loc}$ property (which is implied by the  Ricci lower bound curvature assumption).

\begin{Lem}[\mbox{\cite[Lemma 2.3]{Kan1}}]\label{multi-control}
Let $M^n$ be a complete Riemannian manifold which satisfies the local doubling condition $(DV)_{loc}$.
Then, there exists $\mathcal{M}=\mathcal{M}(\epsilon)$, depending only on $\epsilon$, such that,  for every $X\subset M^n$ an $\epsilon$-discretization of $M^n$, the multiplicity of the covering $\{ B(x,3\epsilon)\} _{x\in X}$
satisfies  $\mathcal{M}(X,3\epsilon) \leq \mathcal{M}$. 
\end{Lem}
\begin{proof}
Note that since $X$ is 
a  $\epsilon$-separated set then $\{B(x,\frac{\epsilon}{2}) \} _{x\in X}$ is a disjoint family of balls. 
Fix some ball $B(x,3\epsilon)$, $x\in X$ and consider the subset $Y\subset X$ consisting of points  $y$ such that  
$B(x,3\epsilon) \cap B(y,3\epsilon) \neq \emptyset$. Therefore, 
$Y\subset B(x,6\epsilon)$ and $\{B(y,\frac{\epsilon}{2}) \} _{y\in Y}$ is a disjoint family of balls
contained in $B(x,6\epsilon + \frac{\epsilon}{2})$.
We  have 
\begin{equation}\label{up-bound}
\sum_{y\in Y}{\rm vol} (B\left(y,\frac{\epsilon}{2}\right)) \leq {\rm vol}\,B\left(x,6\epsilon + \frac{\epsilon}{2}\right) 
\leq {\rm vol}\,B\left(x,7\epsilon\right). 
\end{equation}
For each $y\in Y$, we have $B (x,7\epsilon) \subset B(y,13\epsilon)$ and by the local doubling
assumption  $(DV)_{loc}$ we obtain (using in the last step that $B\left(y,\frac{13\epsilon}{32}\right) \subset B\left(y,\frac{\epsilon}{2}\right)$)
\begin{equation}
{\rm vol} \,B\left(y,13\epsilon\right) \leq C_{\frac{13\epsilon}{2}} {\rm vol} \,B \left(y,\frac{13\epsilon}{2}\right)
\leq \dots \leq C_{\frac{13\epsilon}{2}} C_{\frac{13\epsilon}{4}}\dots C_{\frac{13\epsilon}{32}}{\rm vol} \,B\left(y,\frac{\epsilon}{2}\right).
\end{equation}
This  implies that 
\begin{equation}
{\rm vol} \,B\left(y,\frac{\epsilon}{2}\right) \geq
\left(C_{\frac{13\epsilon}{2}} C_{\frac{13\epsilon}{4}}\dots C_{\frac{13\epsilon}{32}}  \right)^{-1}
{\rm vol}\,B\left(x,7\epsilon\right).
\end{equation}
Hence,
\begin{equation}
\sum_{y\in Y}{\rm vol} \,B\left(y,\frac{\epsilon}{2}\right) \geq |Y| \, 
\left(C_{\frac{13\epsilon}{2}} C_{\frac{13\epsilon}{4}}\dots C_{\frac{13\epsilon}{32}}  \right)^{-1}
{\rm vol}\,B\left(x,7\epsilon\right)\,.
\end{equation}
On the other hand, by (\ref{up-bound}) we have
$\sum_{y\in Y}{\rm vol} \,B\left(y,\frac{\epsilon}{2}\right) \leq {\rm vol} \,B\left(x,7\epsilon\right)$,
therefore, we obtain  
\begin{equation}
|Y| \leq C_{\frac{13\epsilon}{2}} C_{\frac{13\epsilon}{4}}\dots C_{\frac{13\epsilon}{32}}.
\end{equation}
This concludes the proof of the lemma with $\mathcal{M}(X,3\epsilon) = C_{\frac{13\epsilon}{2}} C_{\frac{13\epsilon}{4}}\dots C_{\frac{13\epsilon}{32}}$.

\end{proof}

\begin{Rem}
An obvious consequence of  Lemma \ref{multi-control} is that the covering $\{B(x,\epsilon) \} _{x\in X}$ 
is locally finite. In the proof of Theorem~\ref{discretization}, we will need to work with the cover induced by  $\{B(x,3\epsilon) \} _{x\in X}$, which is obviously locally finite as well.
\end{Rem}

\begin{Rem}\label{discrete-DV}
Let $M^n$ be a complete Riemannian manifold which satisfies the local doubling condition $(DV)_{loc}$ and $(X,\rho,\mu)$
an $\epsilon$-discretization of $M^n$. Then,  $(X,\rho,\mu)$ satisfies the local doubling condition $(DV)_{loc}$.
\end{Rem}

As we recalled in the introduction, 
Kanai and Coulhon-Saloff-Coste described  conditions under which an $\epsilon$-discretization of $M^n$ is roughly
isometric to $(M^n,d,{\rm vol})$. Kanai assumed a lower bound on the Ricci curvature and positivity of the injectivity radius. Coulhon-Saloff-Coste refined Kanai's  result by only requiring that the volume measure satisfies a local doubling condition. 

\medskip

The question of the invariance of polynomial growth under 
a rough isometry has been also worked out by these authors under the same assumptions as above. Coulhon and Saloff-Coste 
proved that if the volume measure is local doubling, then $\alpha$-polynomial growth is invariant under rough isometries. Let us summarize these results which will be used later in this section.

\begin{Thm}[\mbox{\cite[Section 2]{Co-Sa}, \cite[Section 6]{Kan2}}]
\label{rough-discretization}
Let $M^n$ be a complete $n$-dimensional Riemannian manifold and suppose that $M^n$ satisfies the $(DV)_{loc}$ condition. 
Then, for any $\epsilon$-discretization $X$ of $M^n$, the metric measured  space  $(X,\rho,\mu)$ is roughly isometric to $M^n$.
In particular, when $M^n$ satisfies the lower Ricci curvature bound (\ref{eq:boundonRicci}), then the constants 
appearing in the definition of rough isometry (see Definition~\ref{de:roughisometry}) depend only on $n$, $\kappa$ and $\varepsilon$.
\end{Thm}  
\begin{proof}
This Theorem was first proved by Kanai under two assumptions, a lower bound on the  Ricci curvature and 
the positivity of the injectivity radius (\cite[Lemma 3.6]{Kan2}).
Coulhon and Saloff-Coste later observed that the $(DV)_{loc}$ 
condition is sufficient.  More precisely, condition (\ref{rough1}) holds by definition. 
The proof of condition (\ref{rough2}) 
is derived exactly as  the proof of Lemma 2.5 in \cite{Kan2} since the latter relies only on the $(DV)_{loc}$ condition 
(which is a consequence of the 
lower bound on the Ricci curvature). Condition (\ref{rough3}), which was established in
\cite{Kan2} Lemma 3.6 under the assumption that the injectivity radius is positive, is proved in
\cite{Co-Sa}, proposition 2.2,  where only the condition of local doubling of the volume measure is assumed. 

\end{proof}

The following theorem establishes an invariance property of polynomial growth under rough isometries.
\begin{Thm}[\cite{Co-Sa} \rm{Proposition 2.2}]\label{poly-rough}
Let $(X_1, \rho_1, \mu _1)$ and $(X_2, \rho_2, \mu _2)$ be two metric measure spaces satisfying the $(DV)_{loc}$ condition.
Suppose that $\Phi$ is a rough isometry between $(X_1, \rho_1, \mu _1)$ and $(X_2, \rho_2, \mu _2)$.
Then, there exists a constant $C>0$ depending on the constant of the rough isometry and the constant in the local doubling condition, such that 
\begin{equation}\label{volcompar}
C^{-1} \mu _1 (B(x,C^{-1}R)) \leq \mu _2 (B(\Phi (x),R)) \leq C \mu _1 (B(x,CR))
\end{equation}
for any $x\in X_1$ and $R\geq 1$.
In particular, polynomial growth of order $\alpha$ is invariant under rough isometries. 
\end{Thm}

We can therefore deduce the following corollary.

\begin{Cor}\label{poly-discretization}
Let $M^n$ be a complete $n$-dimensional Riemannian manifold and suppose that $M^n$ satisfies the $(DV)_{loc}$ condition
and let $X$ be an $\epsilon$-discretization of $M^n$.
Then 
$M^n$ has polynomial growth of order $\alpha$ if and only if $X$ has polynomial growth of order $\alpha$.
In particular, when $M^n$ satisfies the lower Ricci curvature bound (\ref{eq:boundonRicci})
and the $\alpha$-polynomial growth estimate (\ref{eq:polygrowth}), 
then the $\epsilon$-discretization $X$ has $\alpha$-polynomial growth, $\mu (B(x,R)) \leq v' R^\alpha$ for every $R\geq R_0'$, where the constants $v'$ and $R_0'$ 
depend only on $n$, $v$, $R_0$, $\kappa$ and $\epsilon$.

\end{Cor}
\begin{proof}
Let  $X$ be a $\epsilon$-discretization of $M^n$. By Theorem \ref{rough-discretization}, $M^n$ and $X$ are roughly isometric.
By Remark \ref{discrete-DV},  $X$ satisfies the $(DV)_{loc}$ condition and 
therefore the assertion of the corollary follows upon applying Theorem \ref{poly-rough}.
\end{proof}
Henceforth, we will consider $\epsilon$-discretization subsets of $M^n$  such  that the covering radius is $\epsilon$. 
We need two definitions before stating the main theorem of this section.

\begin{Def}
\label{de:lengthdisgrad} Let $X$ be an $\epsilon$-discretization of $M^n$.
For a function $f:X \to \mathbb R$, the length of the discrete gradient of $f$ is defined by  
\begin{equation}
\label{eq:lendis}
\delta f (x) = \left (\sum _{y\sim x}|f(y)-f(x)|^2\right )^{1/2}.
\end{equation}
\end{Def}

\begin{Def}
\label{de:discreteP} Given $\sigma \geq 1$ and $\beta \geq 1$,
we say that a discrete measured metric space $(X,\rho,\mu)$ satisfies a uniform $(\sigma,\beta, \sigma)$-Poincar\'e inequality
if there exist constants $r_1$, $C=C(\sigma,\beta)$ and $C'\geq 1$ such that for any function $f: X\to \mathbb R$, any $R\geq r_1$ and $x_0\in X$ we have
\begin{equation}
\label{discrete-P0}
\sum_{x\in B(x_{0},R)} |f(x) -f_R |^{\sigma}\mu (x) \leq C R^{\beta} \sum_{x\in B(x_{0},C'R)} (\delta f (x))^{\sigma}\mu (x),
\end{equation}
where $$f_R=f_{B(x_{0},R)}=\frac{1}{\mu (B(x_0,R))}\sum _{x\in B(x_0,R)} \mu(x) f(x).
$$
\end{Def}

\section{A criteria for a manifold to carry a uniform Poincar\'e inequality}
\label{se:se2}
In \cite{Co-Sa}, Coulhon and Saloff-Coste studied a Poincar\'e inequality (\ref{sigma-beta-sigma-P}) with $\beta =1$ and $\sigma =1$,
that is, a $(1,1,1)$-uniform Poincar\'e inequality. 
Nevertheless, the proof for 
an arbitrary $\beta \geq 1$ and $\sigma \geq 1$ works along the same lines of their arguments. 
Following \cite{Co-Sa} and in order to make this paper self-contained,  we will now provide   
the part of the proof of 
Theorem \ref{discretization} which is needed for our application:
If $X$ satisfies a Poincar\'e inequality (\ref{discrete-P0}), then $M^n$ satisfies a Poincar\'e inequality 
(\ref{sigma-beta-sigma-P}).  In addition, we will carefully keep track of the dependencies of the quantities appearing in the proof. The statement of the  theorem is the following.
\begin{Thm}[\cite{Co-Sa}, Proposition 6.10]
\label{discretization} 
\label{th:discrete}
Let $M^n$ be a complete Riemannian manifold which satisfies the local doubling condition, $(DV)_{loc}$, and the local 
Poincar\'e inequality $(\ref{eq:localP})$. Let $X\subset M^n$ be a $\epsilon$-discretization of $M^n$.
Then $M^n$ satisfies the uniform Poincar\'e inequality (\ref{sigma-beta-sigma-P}) if and only if the discretization $(X,\rho,\mu)$ of $M^n$ 
satisfies the uniform Poincar\'e inequality $(\ref{discrete-P0})$.
 \end{Thm}

\begin{proof} Consider a complete Riemannian manifold $M^n$ and an $\epsilon$-discretization $X$ of $M^n$.
We will prove the part of the Theorem which we will later need: 
\medskip

Given a function $\psi : M^n \to \mathbb R$ 
let the function 
$\tilde \psi : X \to \mathbb R$ be defined by

\begin{equation}
\label{eq:makedis}
\tilde \psi (x) = \psi_{B(x,\epsilon)}=\frac{1}{\rm{vol}\  B(x,\epsilon)}\int_{B(x,\epsilon)}\psi (z) \vol(z).
\end{equation}

For $E\subset M^n$ and $F\subset X$, two functions $\psi : M^n \to \mathbb R$, and $f:X\to \mathbb R$ and $\sigma$ a positive integer, we define
\begin{equation}
\label{eq:norm1}
\Vert \psi \Vert _{\sigma,E}= \left(\int_{E} |\psi (z) |^\sigma d\mbox{\rm vol}(z)\right)^{1/\sigma}
\end{equation}
and
\begin{equation}
\label{eq:norm2}
\Vert f \Vert _{\sigma,F}=\left(\sum _{F}|f(x)|^\sigma \mu (x) \right)^{1/\sigma}.
\end{equation}
Let us recall the following lemma which relates the gradients of $\psi$ and $\tilde \psi$. 

\begin{Lem}[\cite{Co-Sa}, Lemma 6.4]
\label{lemma-6.4}
For any  $\sigma\geq 1$, there exist constants $\mathcal{T}$ and $\mathcal{T}'$ 
depending on $\sigma$ and the multiplicity of the covering associated to the $\epsilon$-discretization $X$ of $M^n$, 
such that for 
all smooth functions $\psi: M^n \to \mathbb R$, 
all $R\geq 1$, and all $x\in X$ the following holds 
\begin{equation}
\label{eq:norm3}
\Vert \delta \tilde \psi \Vert _{\sigma,B(x,R)} \leq \mathcal{T} \Vert \nabla \psi \Vert _{\sigma,B(x,\mathcal{T}'R)}.
\end{equation}
\end{Lem}

\medskip

Inequality (\ref{eq:norm3}), the proof of which the authors referred to their Lemma 5.3 (which is not proved), is stated in Coulhon Saloff-Coste. Let us provide a proof of this Lemma.

\begin{proof}
The Lemma is a direct consequence of the following fact: for any $\epsilon >0$, any $x,y\in M^n$ such that 
$d(x,y)\leq 2 \epsilon$ and any $\psi : \M^n \rightarrow \mathbb R$,
\begin{equation}
\label{lemma-5.3}
|\tilde \psi (x) - \tilde \psi (y)| ^\sigma V(x,\epsilon) \leq C \int_{B(x,6\epsilon)} |\nabla \psi |^\sigma,
\end{equation}
where $C := C(n,\sigma, \epsilon, \kappa)$ and $V(x,\epsilon) = {\rm vol}B(x,\epsilon)$.
Indeed, assuming (\ref{lemma-5.3}) we have 
\begin{align}
\label{A1}
\Vert \delta \tilde \psi \Vert ^{\sigma} _{\sigma, B(x,R)} = {}&
\sum_{z \in B(x,R)} |\delta \tilde \psi (z)|^\sigma V(z,\epsilon)  \\
\leq {}& \sum_{z \in B(x,R)} \mathcal M (\epsilon) ^{\sigma -1} \bigg( \sum_{z', \, d(z,z')\leq 2\epsilon} |\tilde \psi (z') - \tilde \psi (z) |^\sigma \bigg) V(z,\epsilon) \nonumber \\
\leq {}& C \mathcal{M} (\epsilon) ^{\sigma -1}  \mathcal{M} (6\epsilon) \int_{B(x,C'R)} |\nabla \psi | ^\sigma,\nonumber
\end{align}
where $\mathcal{M}(\epsilon)$ is the multiplicity of the covering of $M^n$ by balls of radius $\epsilon$ and $C' = 1+6\epsilon$; 
in (\ref{A1}), the first inequality is due to Jensen's inequality and the second 
follows from (\ref{lemma-5.3}). 

\smallskip

Let us conclude the proof of Lemma \ref{lemma-6.4} by proving (\ref{lemma-5.3}). Note that 
\begin{equation}\label{A2}
\tilde \psi (x) - \tilde \psi (y) = \frac{1}{V(x,\epsilon)} \frac{1}{V(y,\epsilon)}\int_{B(x,\epsilon)}\int_{B(y,\epsilon)} (\psi (u) -\psi (z) )du dz.
\end{equation}
By Jensen's inequality, we get
\begin{equation}\label{A3}
|\tilde \psi (x) - \tilde \psi (y)|^\sigma \leq \frac{1}{V(x,\epsilon)} \frac{1}{V(y,\epsilon)}\int_{B(x,\epsilon)}\int_{B(y,\epsilon)} |\psi (u) -\psi (z)| ^\sigma du dz,
\end{equation}
and by Minkowski's inequality we get,
\begin{equation}\label{A4}
|\tilde \psi (x) - \tilde \psi (y)|^\sigma \leq 2^{\sigma -1} \frac{1}{V(x,\epsilon)} \frac{1}{V(y,\epsilon)}\int_{B(x,\epsilon)}\int_{B(y,\epsilon)} 
\big(|\psi (u) -\tilde \psi (x)|^\sigma + |\psi (z) -\tilde \psi (y)| ^\sigma \big)du dz,
\end{equation}
which implies, by the local Poincar\'e inequality,
\begin{equation}\label{A5}
|\tilde \psi (x) - \tilde \psi (y)|^\sigma \leq 2^{\sigma -1} \frac{C}{V(x,\epsilon)} \int_{B(x,3\epsilon)}|\nabla \psi |^\sigma
+ 2^{\sigma -1} \frac{C}{V(y,\epsilon)} \int_{B(y,3\epsilon)} |\nabla \psi |^\sigma,
\end{equation}
where $C=: C(n,\sigma, \epsilon, \kappa)$. We then deduce,
\begin{equation}\label{A6}
|\tilde \psi (x) - \tilde \psi (y)|^\sigma V(x,\epsilon) \leq 2^{\sigma -1} C \bigg( \int_{B(x,3\epsilon)}|\nabla \psi |^\sigma
+\frac{V(x,\epsilon)}{V(y,\epsilon)} \int_{B(y,3\epsilon)} |\nabla \psi |^\sigma \bigg),
\end{equation}
and by local doubling and the fact that $d(x,y) \leq 2\epsilon$,
\begin{equation}\label{A7}
|\tilde \psi (x) - \tilde \psi (y)|^\sigma V(x,\epsilon) \leq C' \int_{B(x,6\epsilon)}|\nabla \psi |^\sigma.
\end{equation}
\end{proof}

\bigskip

We now turn to the proof of the part of the theorem which will be needed for the applications of this paper.

\medskip
Let  
$(X,\rho,\mu)$ be a fixed $\epsilon$-discretization  of $M^n$ satisfying Poincar\'e inequality (\ref{discrete-P0}). After a normalization, which will affect the constants once and  for all, we can assume that $\epsilon =1$.

We  need to prove that 
for a given $\sigma \geq 1$, there exists a constant $C>0$
such that for any  $r_0 >0$, there exists a constant $K$ such that for any smooth function $\psi : M^n \to \mathbb R$ and 
any $R\geq r_0$, Inequality $(\ref{sigma-beta-sigma-P})$ holds, that is

\begin{equation}
\label{eq:ine}
\int _{B(m,R)} |\psi(x)- \psi _R |^{\sigma}d{\mbox{\rm vol}}(x)  \leq K R^{\beta} \int _{B(m,CR)} |\nabla \psi (x)|^{\sigma} d{\mbox{\rm vol}}(x).
\end{equation}

\medskip

To this end, let us consider a smooth function $\psi$ on $M$, numbers $r_0 >0$ and $\sigma \geq 1$, and a  point $m\in M^n$. Let $R$ satisfy  $R\geq r_0$. Let us define 
\begin{equation}\label{R_1}
R_1= \max\{ \epsilon, r_1\} = \max \{1, r_1\},
\end{equation}
where $r_1$ is determined by the discrete Poincar\'e inequality (\ref{discrete-P0}).

\medskip

The radius $R$ can be either less than or equal to $R_1$, or 
larger than or equal to $R_1$. In the following, we will analyze these cases separately.

\smallskip

In the first case, $r_0 \leq R\leq R_1$, the conclusion essentially follows from the local Poincar\'e inequality. Indeed, Inequality 
(\ref{eq:localP}) yields that 

\begin{equation}
\int _{B(m,R)} |\psi(x)-\psi _R |^{\sigma}d{\rm vol}(x) 
\leq C(n, \sigma, R)\,  \int _{B(m,3R)} |\nabla \psi (x)|^{\sigma} d{\rm vol}(x),
\end{equation} 
and thus allows us to conclude that 
\begin{equation}\label{discretization-case1}
\int _{B(m,R)} |\psi(x)-\psi _R |^{\sigma}d{\rm vol}(x) 
\leq K_1 \, R^{\beta} \int _{B(m,3R)} |\nabla \psi (x)|^{\sigma} d{\rm vol}(x).
\end{equation} 
where
\begin{equation}\label{K1}
K_1 := \frac{1}{r_0 ^{\beta}}sup_{r_0\leq R \leq R_1} C(n,\sigma ,R)
\end{equation}
is a constant which depends on $r_0$, $r_1$, as well as  the local Poincar\'e function $C(n,\sigma, R)$ of $M^n$.

\smallskip

We now consider the second case where 
$R\geq R_1$.
Let $\eta \in \mathbb R$ be a constant to be determined later. Let  $\mathds{1}_{U}$ denote that characteristic function of the set $U$. 
Since $B(m,R) \subset \bigcup_{x\in X\cap B(m,R+\epsilon)}B(x,\varepsilon)$, we have

\begin{align}
\label{eq:1}
\int_{B(m,R)} |\psi(z) -\eta|^{\sigma}\vol(z) \leq {}&
\int_{B(m,R)} \sum _{x\in X\cap B(m,R+\epsilon)}|\psi(z) -\eta|^{\sigma} \mathds{1}_{B(x,\epsilon)}(z)\vol(z) \\
\leq {}&\sum _{x\in X\cap B(m,R+\epsilon)}\int_{B(x,\epsilon)} |\psi(z) -\eta|^{\sigma}\vol(z).\nonumber
\end{align}

For any positive numbers $u,t$ and $\sigma$ an integer, Minkowski's  inequality asserts that  
 $$|u-t|^{\sigma} \leq 2^{\sigma -1} (|u|^{\sigma}+|t|^{\sigma}).$$ It then follows that 

\begin{align}
\label{eq:2}
\int_{B(m,R)} |\psi(z) -\eta|^{\sigma}\vol(z) \leq {}&
2^{\sigma -1}\sum _{x\in X\cap B(m,R+\epsilon)}\int_{B(x,\epsilon)} |\psi(z) -\tilde \psi (x)|^{\sigma}\vol(z) {}\\
+{}& 2^{\sigma -1}\sum _{x\in X\cap B(m,R+\epsilon)}\mu(x)|\tilde \psi (x)- \eta|^{\sigma}.\nonumber
\end{align}

\medskip

Let us denote by (I) and (II), the first term and the second term composing the right-hand side of the last inequality, respectively.

\medskip

One can bound (I) from above by using the local Poincar\'e inequality (\ref{eq:localP}) with $R=\epsilon$ and 
\begin{equation}\label{C1}
C_1= 2^{\sigma -1}C(n,\sigma, \epsilon),
\end{equation}
to obtain 
\begin{equation}
(I) \leq C_1 \sum _{x\in X\cap B(m,R+\epsilon)}\int_{B(x,3\epsilon)}|\nabla \psi|^{\sigma} \vol(z)\, .
\end{equation}
By lemma \ref{multi-control}, the multiplicity of the covering $\{ B(x,3\epsilon)\} _{x\in X}$  is bounded by $\mathcal{M}(\epsilon)$. 
Since  $\epsilon \leq R$, we have that  for each $x\in B(m,R+\epsilon)$  $B(x,3\epsilon) \subset B(m, 5R)$. Therefore, we have 
\begin{equation}
\label{eq:5 R}
(I) \leq C_1 \mathcal{M}(\epsilon) \int_{B(m,5 R)}|\nabla \psi|^{\sigma} \vol(z).
\end{equation}

\medskip

We prove that the second term (II) is bounded in the following way. Since $(M,d,\rm{vol})$ and $(X,\rho,\mu)$ are roughly 
isometric, we can choose $x_0\in X$ such that 
$d(m,x_0)\leq \varepsilon$. By choosing $r=R+2\epsilon +1$ and $C_3 = \frac{3\epsilon+1}{\epsilon}=4$, we have (since $R\geq R_1=\max\{\epsilon, r_1\}\geq\epsilon$ and $\epsilon = 1$) 

\begin{equation}
\label{eq:3}
R+2\epsilon \leq r \leq C_3 R = 4R,
\end{equation}
so that

\begin{equation}
\label{eq:4}
X\bigcap B(m,R+\epsilon) \subset B(x_0,r), 
\end{equation}
and for any constant $C$ we have 
 
\begin{equation}
\label{eq;5}
 B(x_0,Cr) \subset B(m,(CC_3 +1) R) = B(m,(4C+1)R).
 \end{equation}
 
In order to apply the discrete Poincar\'e inequality in the context of (II), let us  choose
\begin{equation}
\label{eq:6}
\eta = \tilde{\psi}_r=\frac{1}{\mu (B(x_0,r))}\sum _{x\in B(x_0,r)} \mu(x) \tilde{\psi}(x).
\end{equation}

By assumption, $(X,\rho,\mu)$ satisfies a $(\sigma , \beta, \sigma)$-Poincar\'e inequality (see Definition~\ref{discrete-P0})
and since $1<R_1<R+\epsilon \leq r$,  we obtain 
\begin{equation}
\label{eq:7}
(II) \leq  2^{\sigma -1}C r^{\beta} \sum _{x\in B(x_0,C'r)}\mu(x)|\delta \tilde \psi (x)|^{\sigma}.
\end{equation}

Therefore,  Inequality~(\ref{eq:norm3}) of Lemma \ref{lemma-6.4} and the fact that $r\leq C_3 R = 4R$ imply 
\begin{equation}
\label{eq;8}
(II) \leq 2^{\sigma -1}(C\mathcal{T}^{\sigma} 4^{\beta}) R^{\beta} \int_{B(x_0,C'\mathcal{T}'r)} |\nabla \psi (z)|^{\sigma} \vol(z).
\end{equation}

\smallskip

We now claim that 
\begin{equation}
\label{eq;9}
\int_{B(m,R)} |\psi(z) -\psi _R|^{\sigma}\vol(z) \leq 2^{\sigma} \inf _{\tau \in \mathbb R}\int_{B(m,R)} |\psi(z) -\tau|^{\sigma}\vol(z),
\end{equation}
where $\psi_R=\psi_{B(m,R)}$.

\medskip

Indeed,  for any $\tau\in \mathbb R$ ,by applying Jensen's inequality we have 
\begin{equation}
\label{eq:jensenagain}
\Vert \tau -\psi_R \Vert _{\sigma ,B(m,R)}=
\left(\int_{B(m,R)}\left(\frac{1}{\rm{vol} \, B(m,R)}\Big\vert \int_{B(m,R)}(\psi(y)-\tau) d\mbox{\rm vol}(y) \Big\vert\right)^{\sigma} \vol(z)\right)^{\frac{1}{\sigma}}
\leq \Vert \psi -\tau \Vert _{\sigma , B(m,R)}.
\end{equation}
 Furthermore,  Minkowski's inequality implies that
\begin{equation}
\label{eq:normonball}
\Vert \psi -\psi _R \Vert _{\sigma , B(m,R)}
 \leq
 \Vert  \psi -\tau \Vert _{\sigma ,B(m,R)}  + \Vert \tau -\psi_R \Vert _{\sigma ,B(m,R)}
 \leq
 2 \Vert \psi -\tau \Vert _{\sigma ,B(m,R)},
\end{equation}
hence, the claim follows.
 
\medskip

Now let us define
\begin{equation}\label{C2}
C_2= C_1{\mathcal M}(\epsilon),
\end{equation}
\begin{equation}\label{C4}
C_4=2^{\sigma-1}C{\mathcal T}^{\sigma}C_3^{\beta} = 
2^{\sigma-1}  (C  {\mathcal T}^{\sigma} 4^{\beta})
\end{equation}
and 
\begin{equation}\label{C5}
C_5=\max{\{C_2,C_4\}}.
\end{equation}

\medskip

Inequalities (\ref{eq:5 R}), (\ref{eq;8}) and the claim  imply that
\begin{equation}
\label{eq:10}
\int_{B(m,R)} |\psi(z) -\psi _R|^{\sigma}\vol(z) \leq 2^{\sigma} \left( C_2 \int_{B(m,5R)} |\nabla \psi (z)|^{\sigma} \vol(z)
+C_4 R^{\beta}\int_{B(x_0,C' \mathcal{T}'r)} |\nabla \psi (z)|^{\sigma} \vol(z) \right),
\end{equation}
 hence, by applying (\ref{eq;5}) we obtain 
\begin{equation}
\label{eq:11}
\int_{B(m,R)} |\psi(z) -\psi _R|^{\sigma}\vol(z) \leq 2^{\sigma}C_5 \left( \int_{B(m,5R)} |\nabla \psi (z)|^{\sigma} \vol(z)
+ R^{\beta}\int_{B(m,(C'\mathcal{T}'C_3 +1)R)} |\nabla \psi (z)|^{\sigma} \vol(z) \right).
\end{equation}

\medskip

This implies that  for any ball of radius $R\geq \epsilon$ we have 
\begin{equation}\label{discretization-case2}
\int_{B(m,R)} |\psi(z) -\psi _R|^{\sigma}\vol(z) \leq  K_2 \, R^{\beta}
\int_{B(m,(C'\mathcal{T}'C_3 +5)R)} |\nabla \psi (z)|^{\sigma} \vol(z),
\end{equation}
where 
\begin{equation}\label{K2}
K_2 := 2^{\sigma} \left(\frac{C_5}{\epsilon ^{\beta}} +  C_5 \right).
\end{equation}
Inequalities (\ref{discretization-case1}) and (\ref{discretization-case2}) then give the required $(\sigma,\beta,\sigma)$- uniform Poincar\'e inequality 
\begin{equation}
\label{eq:end}
\int _{B(m,R)} |\psi(x)- \psi _R |^{\sigma}d{\mbox{\rm vol}}(x)  \leq K R^{\beta} \int _{B(m,C''R)} |\nabla \psi (x)|^{\sigma} d{\mbox{\rm vol}}(x),
\end{equation}

 where 
 \begin{equation}\label{K}
 K= {\rm max} \{ K_1 , K_2 \}
 \end{equation}
  and 
 \begin{equation}\label{C}
 C''= C'\mathcal{T}'C_3 +5 = 4 C'\mathcal{T}' +5.
 \end{equation}
 
This concludes the proof of Theorem \ref{discretization}.

\end{proof}

\bigskip

\noindent Bishop-Gromov Comparison Inequality implies that for any complete Riemannian manifold $M^n$ with Ricci curvature bounded 
below ${\rm Ricci}_{M^n} \geq -(n-1)k$, we have 
\begin{equation}
\label{eq:locdoub}
\mbox{\rm vol} \, B(m,2 R))\leq 2^n \exp((n-1)\sqrt{k}2 R)\,\mbox{\rm vol}\, B(m,R)),
\end{equation}
i.e., $M^n$ is locally doubling. By Theorem~\ref{cor-sigma-P_loc}, such a manifold also satisfies the local Poincar\'e inequality. We may therefore state 

\begin{Pro}\label{D_loc-P_loc}
Let $M^n$ be a complete Riemannian manifold with Ricci curvature bounded below, $\rm{Ricci} \geq -(n-1)\kappa$, then 
$M^n$ satisfies the local Poincar\'e inequality (\ref{eq:localP}) 
and the local doubling property with constants depending on $n$ and $\kappa$.
\end{Pro}

\noindent By applying the assertion of Theorem~\ref{th:discrete} we obtain

\begin{Cor}
\label{poincare-ricci}
Let $M^n$ be a complete Riemannian manifold with Ricci curvature bounded below. Then $M^n$ satisfies 
a uniform-$(\sigma, \beta,\sigma)$
Poincar\'e inequality 
(\ref{sigma-beta-sigma-P}) if and only if an $\epsilon$-discretization 
$(X,\rho,\mu)$ of $M$ satisfies the discrete uniform analogue (\ref{discrete-P0}).
\end{Cor}

\section{Poincar\'e inequality for metric measured graphs}
\label{se:graphs}
In  this section, we prove that metric measured graphs which satisfy a certain growth condition, polynomial 
growth,  support discrete versions of Poincar\'e inequalities as (\ref{discrete-P0}). In the 
applications, such graphs serve as discrete approximations to a complete Riemannian manifold. These graphs 
satisfy the conditions needed in order to apply the work in \cite{Co-Sa} to the proof of Theorem~\ref{th:main1}.

\medskip

Let $(X,\rho,\mu)$ be a metric measured graph; $(X,\rho,\mu)$ will be said to have $\alpha$-polynomial growth if inequality (\ref{eq:polygrowth}) holds with respect to the metric $\rho$ and the measure $\mu$. Let $V,E$ denote the set of vertices and (non-oriented) edges of $X$, respectively. We will write $x\sim y$ when $[x,y]\in E$, where $[x,y]$ denotes the directed edge from $x$ to $y$. 
Given a function $u: V \to \mathbb R$, we let $du :E \to \mathbb R$ denote the {\it gradient} of $u$ defined by $du([x,y])=u(y)-u(x)$.
Let us recall that we defined (see Definition~\ref{de:lengthdisgrad})) the {\it length} of the gradient of $u$ at a vertex $x\in V$ to be

\begin{equation}
\label{eq;lengthgrad}
\delta u(x)=\left(\sum_{y\sim x} |u(y)-u(x)|^2\right)^{1/2}.
\end{equation}
Since $X$ is a discrete space, we can integrate any function $g$ on any subset $F\subset V$ with the restriction of $\mu$ to $F$.
For the counting measure on $X$, we define the integration as 
$\int _{F} g(x) = \sum _{x\in F} g(x)$.  

\medskip

We now establish a $(\sigma,\alpha +\sigma -1 ,\sigma)$-Poincar\'e 
inequality of type (\ref{discrete-P0}).

\begin{Thm}
\label{poincare-on-graph1}
Let $(X,\rho,\mu)$ be a metric measured graph with $\alpha$-polynomial growth, namely, for some $R_0>0$ and any $R\geq R_0$, we have $\mu(B(x,R)) \leq v' R^\alpha$.
Then for and $\sigma\geq 1$, 
for any function $u: X \to \mathbb R$, $R\geq R_0$ 
and any ball $B(p,R)\subset X$, we have 
\begin{equation}
\label{eq:thm3.2}
\int _{B(p,R)} |u(x)- u_R |^{\sigma}d\mu(x)  \leq 6^{\sigma -1} v' R^{\alpha +\sigma -1} \int _{B(p,3R)} |\delta u(x)|^\sigma d\mu(x), 
\end{equation} 
where $u_R=u_{B(p,R)}$.
\end{Thm}

\begin{proof}
\noindent 
Let $\gamma_{x,y}$ is a minimizing geodesic joining $x$ to $y$. By the definition of the length of the gradient of $u$,  we have
\begin{equation}
\label{eq:pathestimate1}
|u(x) -u (y)| \leq \int_{x_i\in\gamma_{x,y}}|\delta u(x_i)|.
\end{equation}
We also have 
\begin{equation}
\label{eq:intestimate}
\int_{B(p,R)} | u(x) -u_R |^{\sigma} d\mu(x) = \frac{1}{\mu(B(p,R))^\sigma} 
\int_{B(p,R)} \Big|\int_{B(p,R)}(u(x) -u(y)) d\mu(y)\Big |^{\sigma}d\mu(x).
\end{equation}

Hence, by normalizing the measures involved to have total mass equal to one and then applying Jensen's inequality twice, 
we obtain
\begin{equation}\label{12}
\int_{B(p,R)} | u(x) -u_R |^{\sigma} d\mu(x) \leq  
\frac{1}{\mu(B(p,R))} \int_{B(p,R)}\Big(\int_{B(p,R)}\Big(\int_{\gamma_{x,y}}|\delta u(x_i)|\Big)^{\sigma} d\mu(y)\Big )d\mu(x).
\end{equation}

By applying Jensen's Inequality again to the innermost integral in equation~(\ref{12}) we get 

\begin{equation}\label{13}
\int_{B(p,R)} | u(x) -u_R |^{\sigma} d\mu(x) 
\leq  \frac{1}{\mu(B(p,R))} \int_{B(p,R)}
\int_{B(p,R)} \ell_{x,y}^{\sigma -1} \Big(\int_{\gamma_{x,y}}|\delta u(x_i)|^{\sigma}\Big ) d\mu(y)d\mu(x),
\end{equation}
where $\ell_{x,y}$ is the length of the geodesic segment $\gamma _{x,y}$.
Since $\gamma_{x,y} \subset B(p,3R)$ for any $x,y\in B(p,R)$, it follows  that $\ell_{x,y}\leq 6 R$. Hence,  

\begin{equation}
\label{eq:overlargerball1}
\int_{\gamma_{x,y}}|\delta u(x)|^{\sigma} \leq \int _{B(p,3R)} |\delta u(x)|^{\sigma}. 
\end{equation}

It follows that 
\begin{equation}
\label{22}
\int_{B(p,R)} | u(x) -u_R |^{\sigma} d\mu(x) \leq  (6 R)^{\sigma -1}\,\mu(B(p,R)) 
\int_{B(p, 3R)}|\delta u(x)|^{\sigma}d\mu(x).
\end{equation}

The polynomial growth assumption implies that

\begin{equation}\label{33}
\int_{B(p,R)} | u(x) -u_R |^{\sigma} d\mu(x) \leq 6^{\sigma -1} v' R^{\alpha+ \sigma -1 }
\int_{B(p, 3R)}|\delta u(x)|^{\sigma}d\mu(x).
\end{equation}
This ends the  proof  of the theorem.
\end{proof}

\begin{Rem}
Inequality \ref{eq:overlargerball1} can be stated because $X$ is discrete. Indeed, on a manifold the geodesic $\gamma_{x,y}$ and the ball $B(p,4R)$ would have different dimensions. This inequality may seem crude,  but we will show that it is in fact \emph{optimal}. Indeed, in Section \ref{se:example}, we exhibit an example of a graph $X$ and a function $u : X \rightarrow \mathbb R$, such that for every $R$ and for a given $x_0\in X$, the support of $\delta u$ in the ball $B(x_0,R)$ is a diameter $L$ and for 
most of couples $(x,y)$ in $B(x_0,R)$, the geodesic $\gamma _{xy}$ goes through $L$ so that $\int_{\gamma _{xy}} |\nabla u| ^\sigma \approx \int_{B(x_0,R)} |\nabla u| ^\sigma$.
Consequently, Inequality \ref{eq:overlargerball1} is essentially an equality.
\end{Rem}

\section{Proofs of the main results}
\label{se:proofs}
We start this section by proving Theorem~\ref{th:main1}.  After recalling a few basic 
definitions from the general setting of Riemannian manifolds, we turn to the proof of Corrolary~\ref{horospheres}.
\subsection{A proof of Theorem~\ref{th:main1}.}
Henceforth, we 
let $M^n$ be a complete Riemannian 
manifold with Ricci curvature bounded below $\rm{Ricci} \geq -(n-1)\kappa$,
and polynomial growth of order $\alpha$ ${\rm vol} B(m,R) \leq v r^\alpha$ for every $R\geq R_0$ (see  (\ref{eq:boundonRicci}) and (\ref{eq:polygrowth}), respectively) for some positive constants $v,\alpha$. By Proposition~\ref{D_loc-P_loc},  $M^n$ satisfies the local doubling condition, $(DV)_{loc}$, and a local Poincar\'e inequality. 
We now consider an $\epsilon$-discretization $X$ of $M^n$ with $\epsilon = 1$.

\smallskip

Corollary~\ref{poly-discretization} implies that $X$ has polynomial growth of order $\alpha$, i.e., there exists 
$R_0'$ such that for any $R\geq R_0'$, $\mu (B(x,R)) \leq v' R^\alpha$;
where $v'$ depends on $n$, $v$ and $\kappa$, and $R_0'$ depends on $n$, $R_0$ and $\kappa$.
Let $p\in X$ be arbitrary. By Theorem~\ref{poincare-on-graph1},  for every $R\geq r_1 = R_0'$,  $X$  satisfies the Poincar\'e inequality (\ref{eq:thm3.2}):
\begin{equation}
\int_{B(p,R)} | u(x) -u_R |^{\sigma} d\mu(x) \leq 6^{\sigma -1} v' R^{\alpha+ \sigma -1 }
\int_{B(p,3R)}|\delta u(x)|^{\sigma}d\mu(x).
\end{equation}
\smallskip

Hence, we  are in position to apply Theorem~\ref{discretization} and obtain 
\begin{equation}
\label{eq:end1}
\int _{B(m,R)} |\psi(x)- \psi _R |^{\sigma}d{\mbox{\rm vol}}(x)  \leq K R^{\beta} \int _{B(m,C''R)} |\nabla \psi (x)|^{\sigma} d{\mbox{\rm vol}}(x).
\end{equation}

\smallskip
Let us explicitly summarize what $K$ in the above inequality depends on. 
Recall that  the constant $K$ satisfies 
$K= \rm {max} \{ K_1 , K_2 \}$, $K_1= (1/r_0 ^{\beta} )sup_{r_0\leq R \leq R_1} C(n,\sigma ,R)$ and 
$K_2= 2^{\sigma} \left(\frac{C_5}{\epsilon ^{\beta}} +  C_5 \right)$ where
$C_5=\max \{C_2,C_4\}$, $C_2= C_1{\mathcal M}(\epsilon)$, $C_4=2^{\sigma-1}C{\mathcal T}^{\sigma}C_3^{\beta} = 
2^{\sigma-1} 4^{\beta} C{\mathcal T}^{\sigma}$ (cf. (\ref{K}), (\ref{K1}), (\ref{K2}), (\ref{C5}), (\ref{C2}) and (\ref{C4})) with 
$C= 6^{\sigma -1} v'$ (cf. (\ref{eq:thm3.2})).
The constant $C''$ in (\ref{eq:end1}) has been defined by $C'' = C'\mathcal{T}'C_3 +5 = 4 C'\mathcal{T}' +5$ (cf. (\ref{C})) with $C' =3$.

\smallskip

We deduce that  the constant $K$ depends on $n, \sigma, r_0, R_0, \kappa$ and  $v$, and the constant $C''$ depends on $n, \kappa$. This ends the proof of Theorem \ref{th:main1}.
\eop{}

\subsection{Uniform and global Poincar\'e inequality for horospheres}
We now turn to the proof of  Corollary \ref{horospheres}. 
Let $M^n$ be a $n$-dimensional closed Riemannian manifold with its negative sectional curvature uniformly satisfying
\begin{equation}
\label{eq:sectional}
-a^2 \leq K \leq -b^2<0.
\end{equation}
Let $\tilde{M}^n$ be the universal cover of $M^n$,  $T^1 \tilde M^n$ its unit tangent bundle, and  $\pi : T^1 \tilde M^n \rightarrow \tilde M^n$
the canonical projection. 
We denote by $\partial \tilde M^n$ the ideal boundary of $\tilde M^n$.
For $v\in T^{1}\tilde M^n$, let $\gamma _v(t)$ be the geodesic in $\tilde M^n$ such that $\gamma _v (0) = \pi (v)$ and $\dot \gamma (0) =v$. 
Given a point $\xi = \gamma _{v}(-\infty) \in \partial \tilde M^n$, and a base point $x_0 \in \tilde M^n$, for all $\xi \in \partial \tilde M^n$ and for all $x \in \tilde M^n$, the Busemann function 
$B_\xi(\cdot)$ is then defined 
by $B_{\xi}(x)= \lim_{t\to -\infty} (d(x,\gamma_{v} (t)) - d(x_0, \gamma_{v} (t)))$. 
   it is known that since $M^n$ is a closed negatively curved manifold, for each $\xi \in \partial \tilde M^n$, 
the Busemann function $B_\xi(\cdot)$ is smooth.  Furthermore, 
for any $t\in\mathbb{R}$,
the level set  $H_{\xi}(t) =\left\{ x\in {\tilde M^n};\, B_\xi(x)= t\right\}$ is a smooth
submanifold of $\tilde M^n$ which is diffeomorphic to $\mathbb{R}^{n}$ and is  called a {\it horosphere} centred at $\xi$ (the reader is referred to \cite{Ebe} for the necessary background).
For each $v\in T^{1}\tilde M^n$, let $W^{su}(v)$ denote  the {\it strong unstable leave} of the of the geodesic flow on $T^{1}\tilde M^n$.
Recall that $\pi (W^{su}(v))$ can be identified with the horosphere $H_\xi (0)$ centered at $\xi = \gamma _v (-\infty)$ 
and passing through $\pi (v)$, 
that is $\pi (W^{su}(v)) = H_\xi (0)$. 

\smallskip

For $t\in \mathbb R$, let $\exp_t : W^{su}(v) \to \tilde M^n$ be the restriction of the exponential 
map  to $H_\xi (0)$, i.e.,  for any unit vector $u \in W^{su}(v)$, $\rm{exp}_t(u) = \gamma _u (t)$.

\medskip

\noindent{\bf A proof of Corollary~\ref{horospheres}.}
Let us consider a horosphere,  $H := H_\xi$, centered at $\xi\in \partial \tilde M^n$.  We let $\rho$ denote the distance on $H$ determined by the induced Riemannian metric on $H$.
Let us  prove that $H$, endowed with $\rho$ and the corresponding induced vol measure (which by abuse of notation we will denote by vol),  has the following polynomial growth: For every $R\geq 1$,
\begin{equation}
\label{eq:growth}
{\rm vol} \, B(p,R) \leq D R^{\alpha},\ \text{ with } \alpha=\frac{(n-1)a}{b},
\end{equation}
where $B(p,R)$ is a ball in $H$ centered at $p$ and having radius $R$, and $D$ is a constant.
\medskip

Our starting point is a distance comparison proposition due to  E.~Heintz and H.~ImHof; the proof is a consequence of Rauch's comparison theorem, which can be applied due to 
 the assumption on the  sectional curvature of $\tilde M$.  

\begin{Pro}[\mbox{\cite{He-Im-Ho}, Proposition 4.1}]
Let $u\in T^{1}\tilde M^n$ be a unit tangent vector on $\tilde M^n$ and let $v, \, w \in W^{su}(u)$ be two unit vectors in 
the strong unstable leaf of $u$. Then for all $t\geq 0$, 
the distance between $\gamma _v (t)$ and $\gamma _w (t)$ 
satisfies 
\begin{equation}
e^{bt} \rho( \gamma _v (0), \gamma _w (0))\leq \rho( \gamma _v (t), \gamma _w (t)) 
\leq e^{at} \rho( \gamma _v (0), \gamma _w (0)).
\end{equation}
\end{Pro}

\medskip

The two following properties are immediate consequences of this proposition:
\begin{equation}
B(\gamma _u (t)), e^{bt}) \subset \pi \left( {\rm exp}_t(\pi ^{-1}B(\gamma _u (0),1))\right),
\end{equation} 
and
\begin{equation}
{\rm vol}\,B(\gamma _u (t), e^{bt}) 
\leq {\rm vol}\,\pi \left( {\rm exp}_t(\pi ^{-1} B(\gamma _u (0),1))\right) 
\leq e^{(n-1)at} {\rm vol}\,B(\gamma _u (0),1),
\end{equation}
where $B(\gamma _u (0),1)$ and $B(\gamma _u (t)), e^{bt})$ 
are balls on the horosphere $H_\xi (0)$ and $H_\xi (t)$, respectively, with
$\xi = \gamma _u (-\infty)$.

\medskip

Therefore, if we let $R=e^{bt}\, , t\geq 0$ it follows that 
\begin{equation}
\label{eq:horo-ball}
{\rm vol}\,B(\gamma _u (t), R) \leq R^{\alpha} {\rm vol}\,B(\gamma _u (0),1).
\end{equation}

Consider now the ball $B(x,R)\subset H_\xi (0)$ centered at
$x=\pi v$ the base point of the unit tangent vector $v$.  
Let $u = \dot{\gamma} _v (-t)$ be the unit tangent vector such that 
$\gamma _u (t) =x$ 
From (\ref{eq:horo-ball}) we have
\begin{equation}
{\rm vol}\,B(x, R) \leq R^{\frac{(n-1)a}{b}} {\rm vol}\,B(\gamma _u (0),1).
\end{equation}
The set $\{ (\tilde H,p) \, | \, p\in \tilde H, \tilde H\in \tilde M^n \}$ of pointed horospheres of $\tilde M^n$
is homeomorphic to $T^1 \tilde M^n$ and, therefore, since 
$M$ is closed,  is co-compact. It follows that there exists a positive constant $D$ such that 
\begin{equation}
\label{eq:C}
D =\sup\{{\rm vol}\,B(p,1)\} < \infty,
\end{equation}
where the supremum is taken over all balls of radius $1$ on all horospheres.
We then conclude that  for every $R\geq 1$,
\begin{equation}
\label{eq:x}
{\rm vol}\,B(x, R) \leq D R^{\alpha}.
\end{equation} 

Hence, horospheres in $\tilde M^n$ have uniform polynomial growth.
Furthermore, since $M^n$ is closed, horospheres have uniform
bounded sectional curvature (for the induced Riemannian metric) and in 
particular will have a uniform lower bound on their Ricci curvature.
Therefore, any  horosphere $H$ satisfies Poincar\'e inequality (\ref{sigma-beta-sigma-P}) 
 with $\alpha$ as defined above. \eop{}

\section{Examples}
\label{se:example}

It is natural to ask if the inequalities we derived in Theorem~\ref{th:main1} can be improved.
In this section, we show that the assertions of this theorem is optimal in the sense that, when $\alpha \geq 1$ and $\sigma \geq 1$, 
one can construct a Riemannian manifold $M^n$ of Ricci curvature bounded below and polynomial growth of order $\alpha$ which does not carry a 
$(\sigma ,\beta,\sigma)$ Poincar\'e inequality with $\beta < \alpha + \sigma -1$. In fact, Theorem \ref{poincare-on-graph1} is also optimal: we
will construct a graph of polynomial growth of order $\alpha \geq 1$, which  does not carry a $(\sigma ,\beta,\sigma)$ Poincar\'e inequality with $\beta < \alpha + \sigma -1$, for any $\sigma \geq 1$. Following this, we will construct a manifold $M^n$ that is roughly isometric to the graph. In these examples, we will assume for simplicity that $\alpha =2$, the general case can be done in the same way.

\medskip

To this end, let us  first construct  a planar embedded graph
$G$ with a quadratic growth. Let $\mathbb R^2$ be endowed with the Euclidean metric.
The graph $G$ is the following ``antenna like'' embedded in this $\mathbb R^2$ as 
\begin{equation}
\label{eq:graph}
G= \{x=0\} \cup _{n\in \mathbb Z} \{y=n\}.
\end{equation}
The vertex set $V$ of $G$ is defined by $V=\{ (m,n)\  |\  m,\,n \in \mathbb Z\}$.  An edge of $G$ is either a vertical segment joining $(0,n)$ and $(0,n+1)$, $n\in \mathbb Z$, or a horizontal segment 
joining $(n,m)$ and $(n+1,m)$, $n,\, m \in \mathbb Z$. In particular,  that there is no vertical edge joining
 $(n,m)$ and $(n,m+1)$ for $n \neq 0$.
 
\begin{figure}[htbp]
\begin{center}
\scalebox{.45}{ \input{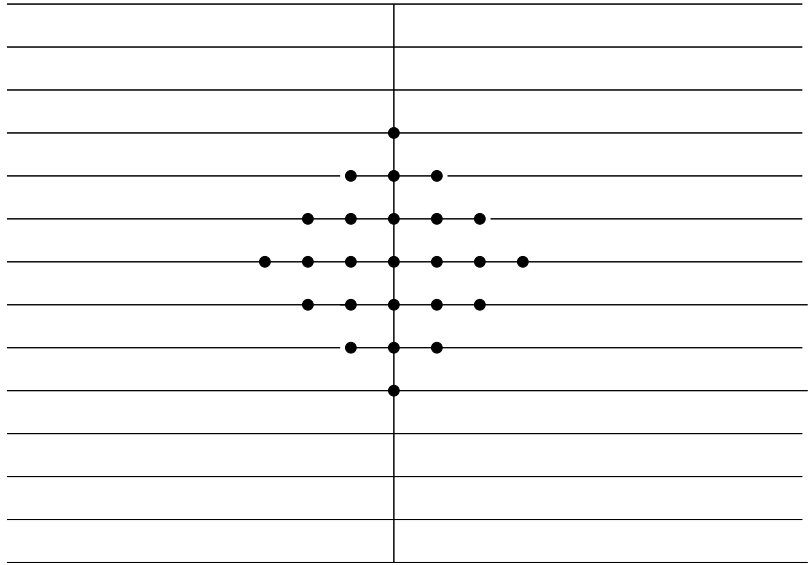}}
\caption{A ball of radius $4$ in $G$.}
\label{figure:quad}
\end{center}
\end{figure}

The distance $d$ on $G$ is the intrinsic distance  induced by the embedding of $G$ in $\mathbb R^2$,
\begin{equation}
\label{eq:dist}
d((m,n),(m',n')) = |m|+|m'|+ |n-n'|,
\end{equation}
and the measure on $G$ is the counting measure.

\smallskip

Given two functions $f,\, g : [0, +\infty[  \rightarrow [0, +\infty[$, we will write $f\asymp g$ if there exists 
a constant $c>0$ such that $f(R) \leq cg(R)$  and $g(R) \leq c f(R)$ for $R$ large enough.
It is easy to check that the volume of an open ball of radius $R$ in $G$  satisfies 
\begin{equation}
\label{eq:ballgrowth}
V(R)=2\big(1+3+\ldots+(2R+1)\big)-(2R+1) \asymp R^2.
\end{equation}

\medskip

\smallskip

We now construct a manifold model for $G$. Consider $G \subset \mathbb R ^2 \subset \mathbb R^3$. For $\epsilon >0$ small enough,
the set $S_\epsilon$ of points in $\mathbb R ^3$ at distance $\epsilon$ from $G$ is a 
smooth surface. The surface $S_\epsilon$ inherits a Riemannian metric 
induced by the metric of $\mathbb R^3$ so that $S_\epsilon$ is made of flat cylinders
attached together at the vertices $(0,n)$, $n\in \mathbb Z$.

\smallskip

Note that the graph and the surface are embedded in ${\mathbb R}^3$ and that
the projection of the surface on the graph is Lipschitz.
Consider now the graph in ${\mathbb R}^2$ (as a horizontal plane in ${\mathbb R}^3$). Then, 
the vertical projection up from the graph to the surface is also Lipschitz.
The first map is surjective and the second map has an image 
whose $2\epsilon$-neighborhood covers the surface $S_\epsilon$, so the
graph and the surface are roughly isometric.

\begin{figure}[htbp]
\begin{center}
\scalebox{.40}{ \input{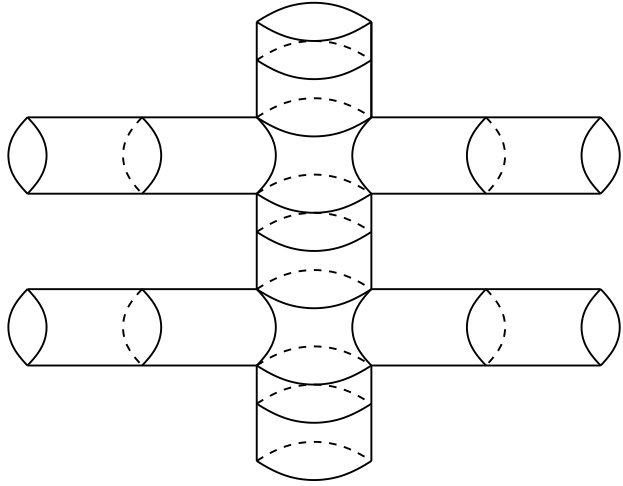}}
 \caption{Part of the surface $S_\epsilon$.}
\label{figure:quad}
\end{center}
\end{figure}

\medskip

\medskip

Let us now define a 
function $u:V\rightarrow \mathbb R$ such that for any positive constant $C>0$ and any $\beta < \alpha + \sigma -1$, $\sigma \geq 1$, 

\begin{equation}
\label{eq:optimal1}
\lim _{R\to \infty} \left( \int _{B(R)} |u- u_{R} |^{\sigma}\right)\left(R^{\beta}\int _{B(CR)} |\nabla(u)|^{\sigma} \right)^{-1} = \infty.
\end{equation}

The function  $u$ is given by 
\begin{equation}\label{eq:function1}
u(m,n)=n, \, \mbox{\rm for all}\  m,n\in\mathbb Z,
\end{equation}
 for any horizontal edge, $u$ is defined to be  
its value on one of the endpoints. Finally, on vertical edges, $u$ is defined by extending its value at the endpoints linearly.

\smallskip

\begin{Lem}
For any positive number $C$,

\begin{equation}
\label{eq:estimate1}
\int_{B(CR)} |\nabla(u)|^{\sigma} \asymp R ,
\end{equation}
and 
\begin{equation}
\label{eq:estimate1}
\int_{B(R)} |u- u_R|^{\sigma}  \asymp R^{\sigma +2},
\end{equation}
where the balls $B(R)$ and $B(CR)$ are centered at $(0,0)$. The relation~(\ref{eq:optimal1}) follows immediately.
\end{Lem}

\begin{proof}
The first estimate follows by a simple counting argument. The second estimate follows by comparing the sum with
(since $u_R=0$) $\int_0^N (2x+1)(N-x)^\sigma dx$. 

\end{proof}

\medskip

With $S_\epsilon$ as defined as above, one argues as before that a relation analogous to ~(\ref{eq:optimal1}) holds. Thus, the assertion of Theorem~\ref{th:main1} is optimal.

\end{document}